\documentclass[11pt]{amsart}
\usepackage{amsmath}
\usepackage{amsthm}
\usepackage{amssymb}
\usepackage{latexsym}
\usepackage{amscd}
\usepackage[all]{xy}

\newtheorem{theorem}{Theorem}  
\newtheorem{lemma}[theorem]{Lemma}
\newtheorem{proposition}[theorem]{Proposition}

\newtheorem{corollary}[theorem]{Corollary}
\newtheorem{remar}[theorem]{Remark}
\renewenvironment{proof}{Proof:\ \ \ }{\QED}
\newenvironment{remark}{\begin{remar}\rm}{\end{remar}}

\newcommand{\QED}{{\unskip\nobreak\hfil\penalty50%
\hskip1em\hbox{}\nobreak\hfil $\Box$%
\parfillskip=0pt \finalhyphendemerits=0 \par\medskip\noindent}}
\newcommand{\bfind}[1]{\index{#1}{\bf #1}}

\newcommand{\n}{\par\noindent}

\newcommand{\sn}{\par\smallskip\noindent}

\newcommand{\bn}{\par\bigskip\noindent}
\newcommand{\pars}{\par\smallskip}
\newcommand{\parm}{\par\medskip}

\newcommand{\im}{\mbox{\rm im}}
\newcommand{\ttiny}{}
\newcommand{\mR}{^{-_{\ttiny R}}}
\newcommand{\pR}{^{+_{\ttiny R}}}
\newcommand{\mF}{^{-_{\ttiny F}}}
\newcommand{\pF}{^{+_{\ttiny F}}}
\newcommand{\ovl}[1]{\overline{#1}}
\newcommand{\subsetuneq}{\mathrel{\raisebox{.8ex}{\footnotesize%
$\displaystyle\mathop{\subset}_{\not=}$}}}

\newcommand{\R}{\mathbb R}

\newcommand{\N}{\mathbb N}
\newcommand{\cX}{\mathcal X}
\newcommand{\cC}{\mathcal C}
\newcommand{\cal}{\mathcal}
\newcommand{\res}{\mbox{\rm res}}

\begin{document}
\title[Spaces of $\R$-places]{Embedding theorems for
spaces of $\R$-places of rational function fields and their products}
\author{Katarzyna and Franz-Viktor Kuhlmann}
\address{Department of Mathematics \& Statistics, University of
Saskatchewan, 106 Wiggins Road, Saskatoon, SK, S7N 5E6, Canada}
\email{fvk@math.usask.ca}
\address{Institute of Mathematics, Silesian University, Bankowa 14,
40-007 Katowice, Poland}
\email{kmk@math.us.edu.pl}
\date{May 13, 2012}
\thanks{The research of the second author was partially supported
by a Canadian NSERC grant.\\
The authors would like to thank the referee who did an amazing job of
carefully reading the manuscript and providing an extensive list of
very helpful corrections and suggestions.}

\begin{abstract}
We study spaces $M(R(y))$ of $\R$-places of rational function fields
$R(y)$ in one variable. For extensions $F|R$ of formally real fields,
with $R$ real closed and satisfying a natural condition, we find
embeddings of $M(R(y))$ in $M(F(y))$ and prove uniqueness results.
Further, we study embeddings of products of spaces of the form $M(F(y))$
in spaces of $\R$-places of rational function fields in several
variables. Our results uncover rather unexpected obstacles to a positive
solution of the open question whether the torus can be realized as a
space of $\R$-places.
\end{abstract}
\maketitle

%
%
%
%
\section{Introduction}
For any field $K$, the set of all orderings on $K$, given by their
positive cones $P$, is denoted by $\cX(K)$. This set is non-empty if
and only if $K$ is formally real. The \bfind{Harrison topology} on
$\cX(K)$ is defined by taking as a subbasis the \bfind{Harrison sets}
\[
H(a)\>:=\>\{P\in \cX(K)\mid a\in P\}\>,\qquad a\in K\setminus \{0\}\>.
\]
With this topology, $\cX(K)$ is a boolean space, i.e., it is compact,
Hausdorff and totally disconnected (see \cite[p.~32]{l2}).

Associated with every ordering $P$ on $K$ is an $\R$-place $\lambda(P)$
of $K$, that is, a place of $K$ with image contained in $\R\cup
\{\infty\}$, which is compatible with the ordering in the sense that
non-negative elements are sent to non-negative elements or $\infty$. The
set of all $\R$-places of $K$ will be denoted by $M(K)$. The Baer-Krull
Theorem (see \cite[Theorem~3.10]{l}) shows that the mapping
\[
\lambda: \cX(K)\longrightarrow  M(K)
\]
(which we will also denote by $\lambda_K$) is surjective. Through
$\lambda$, we equip $M(K)$ with the quotient topology inherited from
$\cX(K)$, making it a compact Hausdorff space (see \cite[p.~74 and Cor.
9.9]{l}), and $\lambda$ a continuous closed mapping. According to
\cite[Theorem~9.11]{l} the subbasis for the quotient topology on $M(K)$
is given by the family of open sets of the form
\[
U(a)\>=\>\{\zeta \in M(K) \mid \zeta(a)>0\}
\]
where $a$ is in the \textbf{real holomorphy ring} of $K$, i.e., $\zeta
(a) \neq \infty$ for all $\zeta \in M(K)$. Since for every $b \in K$ the
element $\frac{b}{1+b^2}$ is in the real holomorphy ring of $K$ (see
\cite[Lemma~9.5]{l}), we have that
\[
H'(b)\>:=\>\{\zeta \in M(K) \mid \infty\ne \zeta(b)>0 \}
\>=\>U\left(\frac{b}{1+b^2}\right)
\]
is a subbasic set for every $b \in K$. So we can assume that the
topology on $M(K)$ is given by the subbasic sets $H'(b)$, $b \in K$.

\pars
Throughout this paper, $R(y)$ will always denote the rational function
field in one variable over the field $R$. For the case of real closed
$R$, we gave in \cite{kmo} a handy criterion for two orderings
on $R(y)$ to be sent to the same $\R$-place by $\lambda\,$:

\begin{theorem}                             \label{glue}
Take a real closed field $R$ and two distinct orderings $P_1,P_2$ of
$R(y)$. Then $\lambda(P_1)=\lambda(P_2)$ if and only if the cuts induced
by $y$ with respect to $P_1$ and $P_2$ in $R$ are upper and lower edge
of a ball in $R$.
\end{theorem}
\n
See Section~\ref{sectcb} for the notions in this theorem and for more
details.

If $R$ is any real closed field, each ordering $P$ on $R(y)$ is
uniquely determined by the cut $(D,E)$ in $R$ where $D=\{d\in R\mid
y-d\in P\}$ and $E=R\setminus D$ (cf.\ \cite{g}). Hence, if $\cC(R)$
is the set of all cuts in $R$, then we have a bijection
\[
\chi:\> \cC(R)\;\longrightarrow\;\cX(R(y))
\]
(which we will also denote by $\chi_R$). With respect to the interval
topology on $\cC(R)$ and the Harrison topology on $\cX(R(y))$, $\chi_R$
is in fact a homeomorphism (see Proposition~\ref{ch}).
Theorem~\ref{glue} can be reformulated as: {\it Two distinct cuts in $R$
are mapped by $\lambda\circ\chi$ to the same place in $M(R(y))$ if and
only if they are upper and lower edge of a ball in $R$.}

In the present paper, we put this result to work in order to find, for
given formally real extensions $F$ of a real closed field $R$,
continuous embeddings $\iota$ of $M(R(y))$ in $M(F(y))$, by finding
suitable embeddings of $\cC(R)$ in $\cC(F)$.

For any field extension $L|K$, the \bfind{restriction}
\[
\res\>=\>\res_{L|K}:\; M(L)\ni\zeta\>\mapsto\> \zeta|_K\in M(K)
\]
is continuous (see \cite[7.2.]{du}). An embedding $\iota:\,M(K)
\rightarrow M(L)$ will be called \bfind{compatible with restriction} if
$\res\circ\iota$ is the identity.

In order to determine when such embeddings of $M(R(y))$ in $M(F(y))$
exist, we have to look at the canonical valuations of the ordered fields
$R$ and $F$. The \bfind{canonical valuation} $v$ of an ordered field is
the valuation corresponding to its associated $\R$-place. If $v$ is the
canonical valuation of the ordered field $F$, then its restriction to
$R$ is the canonical valuation of the field $R$ ordered by the
restriction of the ordering of $F$, and we will denote it again by $v$.
Recall that the ordering and canonical valuation of a real closed field
are uniquely determined. By $vF$ and $vR$ we denote the respective value
groups. Then $vF|vR$ is an extension of ordered abelian groups. Note
that $vR=\{0\}$ if and only if $R$ is archimedean ordered. In
Section~\ref{embM}, we will prove:

\begin{theorem}                             \label{ce}
Take a real closed field $R$ and a formally real extension field $F$ of
$R$. A continuous embedding $\iota$ of $M(R(y))$ in $M(F(y))$ compatible
with restriction exists if and only if $vR$ is a convex subgroup of
$vF$, for some ordering of $F$. In particular, such an embedding always
exists when $R$ is archimedean ordered. If $F$ is real closed, then
there is at most one such embedding.
\end{theorem}
For the case of $F$ not being real closed, we prove a partial uniqueness
result (Theorem~\ref{uniqnotrc}).

Let us point out a somewhat surprising consequence of the above theorem.
If $R$ is a non-archimedean real closed field and $F$ is an elementary
extension (e.g., ultrapower) of $R$ of high enough saturation, then $vR$
will not be a convex subgroup of $vF$ and there will be no such
embedding $\iota$.

\pars
In Section~\ref{embMrc} we consider the special case where $R$ is
archimedean ordered and give a more explicit construction of $\iota$ and
a more explicit proof of the uniqueness. The construction we give is of
interest also when other spaces of places are considered (e.g., spaces
of all places, together with the Zariski topology).

It is well known that for an archimedean real closed field $R$,
$M(R(y))$ is homeomorphic to the circle (over $\R$, with the usual
interval topology). In fact, this is an easy consequence of
Theorem~\ref{glue}. Hence our embedding result shows that each
$M(F(y))$ contains the circle as a closed subspace.

\parm
While spaces of orderings are well understood, this is not the case for
spaces of $\R$-places. Some important insight has been gained (see for
instance \cite{bg}, \cite{br1}, \cite{br2}, \cite{eo}, \cite{gm},
\cite{kmo}, \cite{mmo}, \cite{sch}), but several essential questions
have remained unanswered. For example, it is still an open problem which
compact Hausdorff spaces are realized as $M(F)$ for some $F$. It is
therefore important to determine operations on topological spaces (like
passage to closed subspaces, taking finite disjoint unions, taking
finite products) under which the class of realizable spaces is closed.
It has been shown in \cite{eo} that closed subspaces and finite disjoint
unions of realizable spaces are again realizable, as well as products of
a realizable space with any boolean space.

It has remained an open question whether the product of two realizable
spaces is realizable. A test case is the torus; it is not known whether
the torus (or any other subspace of $\R^n$ of dimension $>1$)
is realizable.


As $M(\R(y))$ is the circle, $M(\R(x))\times M(\R(y))$ is the torus. In
Section~\ref{embtor} we generalize our construction given in
Section~\ref{embMrc} to obtain a natural embedding of $M(\R(x)) \times
M(\R(y))$ in $M(\R(x,y))$. In view of the above-mentioned negative
result, this embedding cannot be continuous with an image that is closed
in $M(\R(x,y))$, because otherwise it would follow from the
realizability of closed subspaces that the torus is realizable. We show
an even stronger negative assertion: the image of the embedding is dense
in, while not being equal to, $M(\R(x,y))$. Hence, the image is not
closed, and the embedding is not continuous.

In the final Section~\ref{sectrp} we will show that for an arbitrary
extension $L|K$, there is a continuous embedding of $M(K)$ in $M(L)$
compatible with restriction as soon as $L$ admits a \bfind{$K$-rational
place}, that is, a place trivial on $K$ with image $K\cup\{\infty\}$. In
particular, this applies when $L$ is a rational function field over $K$.

%
%
%
%
\section{Cuts, balls and $\R$-places}       \label{sectcb}
Take any totally ordered set $T$ and $D,E\subseteq T$. We will
write $D<E$ if $d<e$ for all $d\in D$ and $e\in E$. Note that
$\emptyset<T$ and $T<\emptyset$. For $c\in T$, we will write $c>D$ if
$c>d$ for all $d\in D$, and $c<E$ if $c<e$ for all $e\in E$.

A pair $C=(D,E)$ is called a \bfind{cut in $T$} if $D<E$ and $D\cup
E=T$. In this case, $D$ is an \bfind{initial segment of $T$}, that is,
if $d\in D$ and $d>c\in T$, then $c\in D$; similarly, $E$ is a
\bfind{final segment of $T$}, that is, if $e\in E$ and $e<c\in T$, then
$c\in E$.

We include the cuts $C_{-\infty}=(\emptyset,T)$ and $C_{\infty}=
(T,\emptyset)$; the empty set is understood to be both initial and final
segment of $T$. If $C_1=(D_1,E_1)$ and $C_2=(D_2,E_2)$ are two cuts,
then we will write $C_1<C_2$ if $D_1\subsetuneq D_2\,$.

Take any non-empty subset $A$ of $T$. By $A^+$ we will denote the cut
$(D,T\setminus D)$ for which $D$ is the smallest initial segment of $T$
which contains $A$. Similarly, by $A^-$ we will denote the cut
$(T\setminus E,E)$ for which $E$ is the smallest final segment of $T$
which contains $A$.

A cut $(D,E)$ is called \bfind{principal} if $D$ has a last element or
$E$ has a first element. In the first case, the cut is equal to
$\{d\}^+$, where $d$ is the last element of $D$; in this case we will
denote it by $d^+$. In the second case, the cut is equal to $\{e\}^-$,
where $e$ is the first element of $E$; in this case we will denote it by
$e^-$.

We will need the following fact:
\begin{lemma}                               \label{pred}
If $C_1,C_2$ are cuts in $T$ such that $C_1<C_2$, then
%
%
$C_1\leq a^-<a^+\leq C_2$ for some $a\in T$.
\end{lemma}
\begin{proof}
Write $C_1=(D_1,E_1)$ and $C_2=(D_2,E_2)$. If $C_1<C_2$, then there is
some $a\in D_2\setminus D_1\,$. Then $C_1\leq a^-<a^+\leq C_2\,$.
%
\end{proof}

For any pair $(D,E)$ such that $D<E$, we define the \bfind{between set}
\[
\mbox{\rm Betw}_T (D,E)\>:=\>\{c\in T\mid D<c<E\}\>.
\]

Now consider any ordered field $F$ with its canonical valuation $v$. If
$D,E$ are any subsets of $F$, we set
\[
v(E-D)\>:=\>\{v(e-d)\mid e\in E\,,\, d\in D\}\>\subseteq
vF\cup\{\infty\}\>.
\]
The following observation is easy to prove.
\begin{lemma}                               \label{ifsvis}
Assume that $D$ is an initial segment or $E$ is a final segment of $F$.
Then $v(E-D)$ is an initial segment of $vF\cup\{\infty\}$.
\end{lemma}

A subset $B\subseteq F$ is called a \bfind{ball} in $F$ (with respect
to the valuation $v$) if it is of the form
\[
B\;=\;B_S(a,F)\;:=\;\{b\in F\mid v(a-b)\in S\cup\{\infty\}\}
\]
where $a\in F$ and $S$ is a final segment of $vF$. We consider $S=
\emptyset$ as a final segment of $vF$; we have that $B_\emptyset(a,F)
=\{a\}$.

The notion of ``ball'' does not refer to some space over $F$, but to the
ultrametric underlying the natural valuation of $F$. Note that because
of the ultrametric triangle law, every element of a ball is a center,
that is, if $b\in B_S(a,F)$ then $B_S(a,F)=B_S(b,F)$. Therefore,
$v(b-c)\in S$ for all $b,c\in B_S(a,F)$. A subset $B$ of $F$ is a ball
if and only if for any choice of $a,b\in B$ and $c\in F$ such that
$v(a-c)\geq v(a-b)$ it follows that $c\in B$.

If $0\in B_S(a,F)$, then $B_S(a,F)=B_S(0,F)$ is a convex subgroup of the
ordered additive group of $F$. Every ball in $F$ is in fact a coset of a
convex subgroup: $B_S(a,F)=a+B_S(0,F)$.

By a \bfind{ball complement} for the ball $B=B_S(a,F)$ we will mean a
pair $(D,E)$ of subsets of $F$ such that $D<B<E$ and $F=D\cup B\cup E$.
In this case again, $D$ is an initial segment and $E$ is a final segment
of $F$.

\begin{lemma}                               \label{bcS}
If $(D,E)$ is a ball complement for $B=B_S(a,F)$, then
\[
v(E-D) \>=\> v(E-B) \>=\> v(B-D) \>=\> vF\setminus S\>.
\]
\end{lemma}
\begin{proof}
First, we show that $v(E-D)=vF\setminus S$.
For $d\in D$ and $e\in E$, we have that $v(a-d)<S$ and $v(e-a)<S$
because $d,e\notin B$. From $d<a<e$ it then follows that
$v(e-d)=\min\{v(e-a),v(a-d)\}<S$. This proves that $v(E-D)<S$.

Now take $\alpha\in vF$, $\alpha<S$. Choose $0<c\in F$ such that
$vc=\alpha$. Then $v(a-(a-c))=vc=\alpha$, whence $a-c\notin B$
and therefore, $d:=a-c\in D$. Similarly, $a+c\notin B$ and
therefore, $e:=a+c\in E$. Since $d<a<e$, we find $\alpha=v(2c)=v(e-d)\in
v(E-D)$. Since $v(E-D)$ is an initial segment of $vF\cup\{\infty\}$ by
Lemma~\ref{ifsvis}, and $S$ is a final segment, we can now conclude that
$v(E-D)=vF\setminus S$.

\parm
Again by Lemma~\ref{ifsvis}, also $v(E-B)$ and $v(B-D)$
are initial segments of $vF\cup\{\infty\}$.
%
%
If $d\in D$, $e\in E$ and $b\in B$, then $d<b<e$, whence
$v(b-d)\geq v(e-d)$ and $v(e-b)\geq v(e-d)$. Consequently, $v(E-D)$ is
contained in $v(E-B)$ and $v(B-D)$. On the other hand, $d,e\notin B$
implies that $v(b-d),v(e-b)<S$. So by what we have proved earlier,
$v(b-d),v(e-b)\in v(E-D)$. This shows that all three sets are equal.
\end{proof}

We will say that a cut is the lower edge of the ball $B=B_S(a,F)$ if it
is the cut $B^-$; similarly, a cut is said to be the upper edge of the
ball $B$ if it is the cut $B^+$. Two cuts will be called
\bfind{equivalent} if they are either equal or one is the lower edge
$B^-$ and the other is the upper edge $B^+$ of a ball $B$.

A cut of the form $B^+$ or $B^-$ for $B$ a ball will be called a
\bfind{ball cut}. Principal cuts in $F$ are ball cuts: $a^+=\{a\}^+=
B_{\emptyset}(a,F)^+$ and $a^-=\{a\}^-= B_{\emptyset}(a,F)^-$.

If a cut is neither the lower nor the upper edge of a ball, then we call
it a \bfind{non-ball cut}. The equivalence class of a non-ball cut is a
singleton. As the following lemma will show, the equivalence class of a
ball cut consists of two distinct cuts.

\begin{lemma}                               \label{nuele}
If a cut is the upper or the lower edge of a ball in $F$, then the ball
is uniquely determined. In particular, $B_1^+=B_2^-$ for two balls $B_1$
and $B_2$ is impossible. Therefore, equivalence classes of balls contain
at most two cuts.
\end{lemma}
\begin{proof}
We show the assertion for a cut $B^+=B_S(a,F)^+$; the case of
$B_S(a,F)^-$ is similar.

Take any $d\in F$ and some final segment $T$ of $vF$. Suppose that
$B^+=B_T(d,F)^+$. Since the balls $B_S(a,F)$ and $B_T(d,F)$ are final
segments of the left cut set of $B^+$, their intersection is non-empty.
So one of them is contained in the other. If they were not equal, the
bigger one would contain an element which is bigger than all elements in
the smaller ball, but that is impossible.

Now suppose that $B^+=B_T(d,F)^-$. Then $d>B_S(a,F)$, so $v(a-d)<S$.
Similarly, $a< B_T(d,F)$, so $v(a-d)<T$. Set $d':=(d+a)/2$; then
$d<d'<a$ and $v(a-d')= v(a-d)=v(d'-d)$. Consequently, $d'>B_S(a,F)$
and $d'<B_T(d,F)$, a contradiction.
\end{proof}

In combination with Theorem~\ref{glue}, this lemma shows that the mapping
$\lambda$ will glue not more than two orderings into one $\R$-place. The
other, quite different way of proof is by an application of the
Baer-Krull Theorem.

\begin{proposition}
Take a real closed field $F$. Then for every $\zeta\in M(F(y))$, the
preimage $\lambda^{-1}(\zeta)$ consists of at most two orderings.
\end{proposition}

\pars
Let us add the following observation:

\begin{proposition}                           
For every formally real field $F$, the mapping $\lambda:\cX(F)\rightarrow
M(F)$ induces \bfind{continuous glueings}, that is, if $P_1,P_2\in
\cX(F)$ such that for every pair of open neighborhoods $U_1$ of $P_1$
and $U_2$ of $P_2$ there are $Q_1\in U_1$ and $Q_2\in U_2$ with
$\lambda(Q_1)=\lambda(Q_2)$, then $\lambda(P_1)=\lambda(P_2)$.
\end{proposition}
\begin{proof}
Take two orderings $P_1,P_2\in \cX(F)$ such that $\lambda(P_1)\ne
\lambda(P_2)$. Since $M(F)$ is Hausdorff, there are disjoint open
neighborhoods $U'_1$ of $\lambda(P_1)$ and $U'_2$ of $\lambda(P_2)$.
Their preimages $U_1:=\lambda^{-1}(U'_1)$ and $U_2:=\lambda^{-1}(U'_2)$
are open neighborhoods of $P_1$ and $P_2\,$, respectively. Since $U_1
\cap U_2= \emptyset$, there cannot exist any orderings $Q_1\in U_1$ and
$Q_2\in U_2$ such that $\lambda(Q_1)=\lambda(Q_2)$.
\end{proof}

%
%
%
%
\section{Topologies on $\cC(F)$ and $\cX(F)$}
Take any ordered field $F$. We have already defined the ordering on
$\cC(F)$. Intervals are defined as in any other linearly ordered set.
Note that the linear order of $\cC(F)$ has endpoints $C_{\infty}$ and
$C_{-\infty}\,$.


The \bfind{interval topology} on $\cC(F)$ (like on every other linearly
ordered set with endpoints) is defined by taking as basic open sets all
intervals of the form $(C_1,C_2)=\{C\in \cC(F)\mid C_1<C<C_2\}$ for any
two cuts $C_1,C_2\in \cC(F)$, together with $(C_1,C_{\infty}]$ if
$C_1\ne C_{\infty}$, and $[C_{-\infty},C_2)$ if $C_{-\infty}\ne C_2$.

Note that in the interval topology on $\cC(F)$, an open interval may have
a first or a last element different from $C_{\infty},C_{-\infty}$.
Indeed, if $C=a^+$ is a principal cut and $C_1<a^-$, then $(C_1,a^+)$
has last element $C$. Similarly, if $C=a^-$ and $a^+<C_2$, then
$(a^-,C_2)$ has first element $C$. However, this is the only way in
which first and last elements will arise in open intervals:

\begin{lemma}                               \label{oi}
Take an interval $I$ that is open in the interval topology. If $C$ is
the first element of $I$, then $C=a^+$ for some $a\in F$. If $C$ is the
last element of $I$, then $C=a^-$ for some $a\in F$.
\end{lemma}
\begin{proof}
A finite intersection or arbitrary union of intervals of the form
$(C_1,C_2)$ will only have a first or last element if that is already
true for one of the intervals. Suppose that $C$ is the first element of
$I$; the case of $C$ being the last element is similar. Then
$C$ is the first element of an interval $(C_1,C_2)$, which means that
there is no cut properly between $C_1$ and $C$. Therefore, our assertion
follows from Lemma~\ref{pred}.
\end{proof}

Let us also note that Lemma~\ref{pred} implies:
\begin{lemma}                               \label{pcdens}
The principal cuts lie dense in $\cC(F)$.
\end{lemma}

\pars
A subset of $\cC(F)$ will be called \bfind{full} if it is closed under
equivalence. We define the \bfind{full topology} on $\cC(F)$ to
consist of all full sets that are open in the interval topology.
This topology is always strictly coarser than the interval topology
because in the latter there are always open sets containing $C_{\infty}$
without containing $C_{-\infty}\,$. Hence it is not Hausdorff, but it is
quasi-compact.

\begin{proposition}                       \label{fullballs}
Let $B$ be a ball in F. Then the intervals $[B^-,B^+]$, $(B^-,B^+)$ and
their complements are full in $\cC(F)$.
\end{proposition}
\begin{proof}
Take any ball $B_1$ in $F$. If $B_1\cap B=\emptyset$, then both $B_1^+$
and $B_1^-$ lie in the complement of $(B^-,B^+)$, and by
Lemma~\ref{nuele}, also in the complement of $[B^-,B^+]$.

If $B_1\cap B\ne\emptyset$, then $B_1\subseteq B$ or $B\subsetuneq
B_1\,$. In the latter case again, both $B_1^+$ and $B_1^-$ lie in the
complements of $[B^-,B^+]$ and $(B^-,B^+)$. If $B_1\subsetuneq B$, then
both $B_1^+$ and $B_1^-$ lie in $[B^-,B^+]$ and in $(B^-,B^+)$. Finally,
if $B_1=B$, then both $B_1^+$ and $B_1^-$ lie in $[B^-,B^+]$ and in the
complement of $(B^-,B^+)$.
\end{proof}

Let us also observe:
\begin{lemma}                               \label{contresC}
If $F|R$ is an extension of ordered fields, then the restriction mapping
$\res: \cC(F)\rightarrow \cC(R)$ preserves $\leq$ and equivalence and is
continuous in both the interval and the full topology. The preimage of
every full subset of $\cC(R)$ under $\res$ is again full.
\end{lemma}
\begin{proof}
It is clear that $\res$ preserves $\leq$. Hence, the preimage of every
convex set in $\cC(R)$ is convex in $\cC(F)$. Therefore, if $I$ is an open
interval in $\cC(R)$, then its preimage $I'$ is convex, and if it has no
smallest and no largest element, then it is open. If it has a smallest
element $C'$, then $\res (C')$ is the smallest element of $I$, hence
equal to $C_{-\infty}$ in $\cC(R)$. Therefore, $I'$ contains the cut
$C_{-\infty}$ of $\cC(F)$, whence $C'=C_{-\infty}$. Similarly, a largest
element of $I'$ can only be equal to $C_{\infty}$ in $\cC(R)$. It
follows that $I'$ is open. We have proved that $\res$ is continuous with
respect to the interval topology.

Suppose that $B$ is a ball in $F$. Then $B_0=B\cap R$ is either empty or
a ball in $R$. In the first case, $\res B^-=\res B^+$, and in the second
case, $\res B^-=B_0^-$ and $\res B^+=B_0^+$. This proves that
$\res$ preserves equivalence. This implies that the preimage $U'$ of a
full set $U$ is again full: if $C_1\in U'$ is equivalent to $C_2$, then
$\res (C_1)\in U$ and $\res (C_2)$ are equivalent, whence $\res (C_2)\in
U$ and $C_2\in U'$. From this and the continuity shown above it follows
that $\res$ is continuous with respect to the full topology.
\end{proof}

Take any ordered field $L$. The notion ``full'' was introduced in
\cite{H} for $\cX(L)$, but only for the Harrison sets. We generalize
the definition to arbitrary subsets $Y$ of $\cX(L)$ by calling $Y$
\bfind{full} if $\lambda_L^{-1} (\lambda_L(Y)) =Y$. We will call two
orderings $P_1,P_2\in \cX(L)$ \bfind{equivalent} if $\lambda(P_1)=
\lambda(P_2)$. Hence, $Y$ is full if and only if it is closed under
equivalence.

Note that the intersection of finitely many full sets is again a full
set and the union of any family of full sets is also a full set. We
define the \bfind{full topology} on $\cX(L)$ by taking as open sets all
full sets that are open in the Harrison topology. In general, this
topology is strictly coarser than the Harrison topology and hence not
Hausdorff, but it is always quasi-compact.

\begin{remark}                         \label{full}
1) \ If $Y$ is a full open (or closed) subset of $\cX(L)$, then
$\lambda(Y)$ is an open (or closed, respectively) subset of $M(L)$.
\sn
2) \ For any $U \subset M(L)$, $\lambda^{-1}(U)$ is a full subset of
$\cX(L)$.
\sn
3) \ Take any extension $L|K$ of ordered fields. Then in the diagram
\begin{displaymath}
{
\begin{array}{ccc}
 \mathcal X(L)& \stackrel {\lambda_L}{\longrightarrow }   & M(L)\\\\
{\res}{ \Big {\downarrow}} & &{\res}{\Big  \downarrow} \\\\
 \mathcal X(K)&\stackrel {\lambda_K}{\longrightarrow }     & M(K)
\end {array}
}
\end{displaymath}
the restriction mappings are continuous, and the diagram
commutes (see \cite[7.2.]{du}). Being continuous
mappings from compact spaces to Hausdorff spaces, the
restriction mappings are also closed and proper.
\end{remark}

The analogue of Lemma~\ref{contresC} is:
\begin{lemma}                               \label{contresX}
If $L|K$ is an extension of ordered fields, then the restriction mapping
$\res: \cX(L)\rightarrow \cX(K)$ preserves equivalence and is continuous
w.r.t.\ both the Harrison and the full topology. The preimage of every
full set in $\cX(R)$ under $\res$ is again full.
\end{lemma}
\begin{proof}
The continuity in the Harrison topology has just been stated. The fact
that $\res$ preserves equivalence follows from the commutativity of the
above diagram. As in the proof of Lemma~\ref{contresC}, this implies the
last assertion, and it follows that $\res$ is also continuous with
respect to the full topology.
\end{proof}

\parm
If $R$ is any real closed field, each ordering $P$ on $R(y)$ is
uniquely determined by the cut $(D,E)$ in $R$ where $D=\{d\in R\mid
y-d\in P\}$ and $E=R\setminus D$ (cf.\ \cite{g}). Hence, we have a
bijection
\[
\chi:\> \cC(R)\;\longrightarrow\;\cX(R(y))\;,
\]
which we will also denote by $\chi_R\,$.

\begin{proposition}                         \label{ch}
With respect to the interval topology on $\cC(R)$ and the Harrison
topology on $\cX(R(y))$, $\chi$ is a homeomorphism. The same holds with
respect to the full topologies. For $C_1,C_2\in\cC(R)$, $C_1$ is
equivalent to $C_2$ if and only if $\chi(C_1)$ is equivalent to
$\chi(C_2)$.
\end{proposition}
\begin{proof}
The first assertion is a consequence of \cite[Prop.2.1]{kmo}. For the
proof of the second assertion, we first prove the third. By definition,
$\chi(C_1)$ is equivalent to $\chi(C_2)$ if and only if $\lambda(\chi
(C_1)) =\lambda(\chi(C_2))$. But by Theorem~\ref{glue}, this holds if
and only if $C_1$ and $C_2$ are equivalent. It follows that the image of
a full subset of $\cC(R)$ under $\chi$ is again full, and the preimage
of a full subset of $\cX(R(y))$ under $\chi$ is again full. Now the
second assertion follows from the first.
\end{proof}

This proposition, together with Theorem \ref{glue}, gives us a
description of $M(R(y))$ as the quotient space of $\cC(R)$ with respect to
the equivalence relation for cuts:

\begin{proposition}                         \label{lc}
Via the mapping $\lambda\circ\chi$, the space $M(R(y))$ with the topology
induced by the Harrison topology is the quotient space of $\cC(R)$ with
the full topology, where the quotient is taken modulo the equivalence of
cuts. The full topology is the coarsest topology on $\cC(R)$ for which
$\lambda\circ\chi$ is continuous. The image of a full open set in $\cC(R)$
under $\lambda\circ\chi$ is open.
\end{proposition}

A place in $M(R(y))$ is called \bfind{principal} if it is the image
under $\lambda\circ\chi$ of a principal cut in $\cC(R)$. From
Proposition~\ref{lc} and Lemma~\ref{pcdens} we obtain:
\begin{lemma}                               \label{ppdens}
The principal cuts lie dense in $M(R(y))$.
\end{lemma}

\pars
We will also need:
\begin{proposition}                         \label{diag}
The restriction mappings in the following diagram are continuous
(w.r.t.\ the interval and the Harrison topology as well as w.r.t.\
the full topologies), and the diagram commutes:
\begin{displaymath}
{
\begin{array}{ccccc}
 \mathcal \cC(F)& \stackrel {\chi_F}{\longrightarrow }
& \mathcal X(F(y))& \stackrel {\lambda_{F(y)}}{\longrightarrow }
& M(F(y))\\\\
{\res}{ \Big {\downarrow}}\mbox{ \; } & &{\res}{\Big\downarrow}\mbox{ \; }
 & &{\res}{\Big  \downarrow}\mbox{ \; } \\\\
 \mathcal \cC(R)& \stackrel {\chi_R}{\longrightarrow }
& \mathcal X(R(y))&\stackrel {\lambda_{R(y)}}{\longrightarrow} & M(R(y))
\end {array}
}
\end{displaymath}
\end{proposition}
\begin{proof}
In view of Lemmas~\ref{contresC} and~\ref{contresX} and part 3) of
Remark~\ref{full}, it just remains to prove that the square on the left
hand side of the diagram commutes. This follows from the fact that the
cut induced by $y$ in $R$ under the restriction of some ordering from
$F(y)$ is simply the restriction of the cut induced by $y$ in $F$ under
this ordering.
\end{proof}

We note the following fact, which is straightforward to prove:
\begin{lemma}                               \label{ipri}
If $\iota$ is an embedding of $\cC(R)$ in $\cC(F)$, or of $\cX(K)$ in
$\cX(L)$, or of $M(K)$ in $M(L)$, compatible with restriction, then the
preimage of a set $U$ under $\iota$ is equal to its image under
restriction.
\end{lemma}

%
%
%
%
\section{Embeddings of $\cC(R)$ in $\cC(F)$}    
We consider an extension $F|R$ of ordered fields. Our goal is to
construct an embedding $\iota$ of $M(R(y))$ in $M(F(y))$ under suitable
assumptions on the extension; this will be done in Section~\ref{embM}.
%
%
In view of Proposition~\ref{lc}, we first define an order preserving
embedding of $\cC(R)$ in $\cC(F)$.

From now on we will frequently have
to compare cuts in $R$ with cuts in $F$. If $C=(D,E)$ is a cut in $R$,
then we will say that the element $a\in F$ {\bf fills} $C$ if $D<a<E$
holds in $F$. Since a subset $A$ of $R$ is also a subset of $F$, we will
have to distinguish whether $A^+$ and $A^-$ are taken in $R$ or in $F$.
If this is not clear from the context, we will write $A\pR$ and
$A\mR$ for the former, and $A\pF$ and $A\mF$ for the latter.

To find a suitable embedding of $\cC(R)$ in $\cC(F)$, we need to study
the set of all elements in $F$ that fill a cut in $R$. More generally,
we have to consider the following situation.

\begin{lemma}                               \label{cutfill}
Take two non-empty sets $D<E$ in $R$. Assume that $(D,E)$ is either a
non-ball cut in $R$ with $\mbox{\rm Betw}_F (D,E)\ne \emptyset$, or a
ball complement in $R$. Then
\[
\mbox{\rm Betw}_F (D,E)\>=\>B_{S}(a,F)
\]
for each $a\in \mbox{\rm Betw}_F (D,E)$, where $S$ is the largest final
segment of $vF$ disjoint from $v(E-D)$ (or equivalently, the largest
subset of $vF$ such that $S>v(E-D)$).
\end{lemma}
\begin{proof}
First, we show that ${\cal B}:= \mbox{\rm Betw}_F (D,E)$ is contained in
$B_S(a,F)$. Take any $d\in D$, $e\in E$ and $b\in\cal B$. As $d<a<e$ and
$|a-b|<e-d$, we have that $v(a-b)\geq v(e-d)$. We show that we must have
$v(a-b)>v(e-d)$, which yields that $b\in B_S(a,F)$.

Suppose that $v(a-b)=v(e-d)$. We assume that $b<a$; the case of $b>a$ is
symmetrical. Then it follows that $v(a-d)=v(e-d)$ and $v(b-d)\geq
v(e-d)$, so that $v\left( \frac{a-b}{e-d} \right)=0$,
$v\left(\frac{a-d}{e-d}\right)=0$ and $v\left(\frac{b-d}{e-d}
\right)\geq 0$. We consider the residues under $v$, which are real
numbers. Firstly, $v\left(\frac{a-b}{e-d}\right)=0$ and
$\frac{a-b}{e-d}>0$ imply that $\left(\frac{a-b}{e-d}\right)v>0$, and
$v\left(\frac{b-d}{e-d}\right)\geq 0$ and $\frac{b-d}{e-d}>0$ imply
that $\left(\frac{b-d}{e-d}\right)v\geq 0$. Secondly, we have that
\[
0\>\leq\>\left(\frac{b-d}{e-d}\right)v
\><\>\left(\frac{a-d}{e-d}\right)v\>,
\]
where the last inequality holds because $\left(\frac{a-d}{e-d}\right)v
-\left(\frac{b-d}{e-d}\right)v=\left(\frac{a-d}{e-d}-\frac{b-d}{e-d}
\right)v=\left(\frac{a-b}{e-d}\right)v>0$. So there are rational numbers
$q_1,q_2>0$ such that
\[
\left(\frac{b-d}{e-d}\right)v\><\>q_1\><\>q_2\><\>
\left(\frac{a-d}{e-d}\right)v\,,
\]
which yields
\[
b-d \><\>q_1(e-d)\><\>q_2(e-d)\><\>a-d\;,
\]
whence
\[
b\><\>d+q_1(e-d)\><\>d+q_2(e-d)\><\>a\;.
\]
Consequently, $d+q_1(e-d)\,,\,d+q_2(e-d)\in \mbox{\rm Betw}_R (D,E)$,
which can only happen in the ball complement case. In this case,
$\mbox{\rm Betw}_R (D,E)$ is a ball $B_{S_0}(a_0,R)$ in $R$, with
$D<a_0<E$. By Lemma~\ref{bcS}, $S_0=vR\setminus v(E-D)$. But
\[
v(d+q_2(e-d)\,-\,(d+q_1(e-d)))\>=\>v((q_2-q_1)(e-d))\>=\>v(e-d)<S_0\;,
\]
in contradiction to $d+q_1(e-d)\,,\,d+q_2(e-d)\in B_{S_0}(a_0,R)$.
We have now proved that $\cal B$ is contained in $B_{S}(a,F)$.

\parm
It remains to show that $B_{S}(a,F)$ is contained in $\cal B$. If this
were not the case, then for some $b\in B_{S}(a,F)$ there would exist
some $d\in D$ with $b\leq d$, or some $e\in E$ with $b\geq e$. We will
assume the first case and deduce a contradiction; the second case is
symmetrical. Since $b\leq d<a$ and $B_{S}(a,F)$ is convex, we have that
$d\in B_{S}(a,F)$.

First, we consider the case of $(D,E)$ being the complement of a ball
$B_{S_0}(a_0,R)$ in $R$. We have that $a_0\in {\cal B}\subseteq
B_{S}(a,F)$, so $B_{S}(a,F)=B_{S}(a_0,F)$. Further, we know from
Lemma~\ref{bcS} that $v(E-D)=vR\setminus S_0\,$. By our choice of $S$,
this implies that $S\cap vR=S_0$, and we obtain that $d\in B_{S}(a_0,F)
\cap R= B_{S_0}(a_0,R)$, a contradiction.

In the non-ball case, we use that $B_{S}(a,F)=B_{S}(d,F)$ to obtain that
$B_{S}(a,F)\cap R=B_{T}(d,R)$, where $T:=S\cap vR$. From $S>v(E-D)$ it
follows that $B_{T}(d,R)<E$. In the present case, $\mbox{\rm Betw}_R
(D,E) = \emptyset$, so we find that $B_{T}(d,R)$ is contained in $D$.
Since $B_{S}(a,F)$ is convex and contains $a>D$, it follows that
$B_{T}(d,R)$ is a final segment of $D$. But this contradicts our
assumption that $(D,E)$ is a non-ball cut.
\end{proof}

\begin{remark}                              \label{remS}
In the case where $(D,E)$ is the complement of a ball $B_{S_0}(a_0,R)$
in $R$, we can choose $a=a_0$. Moreover, $S$ is then equal to the
largest final segment of $vF$ disjoint from $vR\setminus S_0$ (or
equivalently, the largest subset of $vF$ such that $S>vR\setminus S_0$).
\end{remark}

\pars
The next lemma tells us which cuts in $F$ restrict to the same cut in
$R$:

\begin{lemma}                               \label{rescut}
Take any cut $C$ in $R$.
\sn
a) \ If $C=(D,E)$, then the set of all cuts in $F$ that restrict to $C$
is $\{C'\in \cC(F)\mid D\pF\leq C'\leq E\mF\}$. (If $D=\emptyset$,
then $D\pF$ means the cut $F\mF$, and if $E=\emptyset$, then
$E\mF$ means the cut $F\pF$.)
\sn
b) \ Assume that $C=B_0\pR$ or $C=B_0\mR$ for a ball $B_0=B_{S_0}
(a_0,R) \ne R$ in $R$, and take the ball $B_{S}(a_0,F)$ as in
Lemma~\ref{cutfill}. Then the set of all cuts in $F$ that restrict to
the cut $C$ in $R$ is $\{C' \in \cC(F)\mid B_0\pF\leq C'\leq B_{S}
(a_0,F)\pF\}$ for $C=B_0\pR$, and $\{C'\in \cC(F)\mid B_{S}
(a_0,F)\mF \leq C'\leq B_0\mF\}$ for $C=B_0\mR$. If $vR$ is a
convex subgroup of $vF$ and $C$ is not principal, then $B_0\pF=B_{S}
(a_0,F)\pF$, $B_0\mF=B_{S}(a_0,F)\mF$, and the above sets are
singletons.
\end{lemma}
\begin{proof}
The proof of part a) is straightforward. Now assume the hypotheses of
part b). We prove the assertions for $C=B_0\pR$. For $C=B_0\mR$, the
proof is symmetrical. If $(D,E)$ is the ball complement of $B_0$ in $R$,
then $C= (D\cup B_0,E)$. By Lemma~\ref{cutfill}, $\mbox{\rm Betw}_F
(D,E)= B_{S}(a,F)$, which implies that $\mbox{\rm Betw}_F (D\cup
B_0,E)=\{b\in B_{S}(a,F)\mid b>B_0\}$. This implies the first assertion
of part b).

For the proof of the second assertion, assume that $vR$ is a convex
subgroup of $vF$ and that $C$ is not principal. Then $S_0$ is a
non-empty final segment of $vR$, and $S_0\ne vR$ since $B_{S_0}
(a_0,R)\ne R$ by assumption. We wish to show that $S_0$ is an initial
segment of $S$. Since $S_0$ is a final segment of $vR$ and $vR$ is
convex in $vF$, also $S_0$ is convex in $vF$. Hence if $S_0$ were not an
initial segment of $S$, then there would be an element $\gamma\in S$
such that $\gamma<S_0$. On the other hand, $S>vR\setminus S_0$, whence
$S_0>\gamma>vR\setminus S_0\ne\emptyset$. But this contradicts the
convexity of $vR$ in $vF$.

Since $S_0$ is an initial segment of $S$, the ball $B_{S_0}(a_0,R)$ is
coinitial and cofinal in the ball $B_{S}(a_0,F)$. This yields that
$B_0\pF=B_{S}(a_0,F)\pF$ and $B_0\mF=B_{S}(a_0,F)\mF$.
\end{proof}

We define an order preserving embedding $\tilde\iota$ of $\cC(R)$ in
$\cC(F)$ as follows. Take a cut $C$ in $R$. If $C=(D,E)$ is a non-ball
cut in $R$, then we set $\tilde\iota(C)=D\pF$ or $\tilde\iota(C)=E\mF$.
If $C$ is the lower or upper edge of a ball $B_0\ne R$ in $R$ and
$(D,E)$ is the ball complement of $B_0\,$, then we set $\tilde\iota(C)=
D\pF$ if $C=B_0\mR$ is the lower edge, and $\tilde \iota(C)= E\mF$ if
$C=B_0\pR$ is the upper edge. Finally, we set $\tilde \iota(R\mR)= R\mF$
and $\tilde\iota(R\pR)=R\pF$. Note that $\tilde\iota$ is uniquely
determined by this definition if and only if no non-ball cut $(D,E)$ in
$R$ is filled in $F$ because then $D^+=E^-$ will still hold in $F$.

\begin{remark}                              \label{iCnbc}
For a cut $C$ in $R$, its image $\tilde\iota(C)$ is a non-ball cut in
$F$ if and only if $C$ is a non-ball cut in $R$ that is not filled in
$F$. Hence if $\tilde\iota(C)$ is a non-ball cut in $F$ then it is the
only cut in $F$ that restricts to $C$.

Indeed, if $C$ is a ball cut in $R$, then by our definition of $\tilde
\iota$, also $\tilde\iota(C)$ is a ball cut. If $C=(D,E)$ is a non-ball
cut in $R$ that is filled in $F$, then by Lemma~\ref{cutfill}, $D\pF=
B\mF$ and $E\mF=B\pF$ for a ball $B=B_S(a,F)$ in $F$, so $\tilde \iota
(C)$ is again a ball cut. But if the non-ball cut $C=(D,E)$ is not
filled in $F$, then it is also a non-ball cut in $F$, as the restriction
to $R$ of a ball cofinal in the left or coinitial in the right cut set
in $F$ would be a ball in $R$ cofinal in $D$ or coinitial in $E$.
\end{remark}

The embedding $\tilde\iota$ is order preserving since
%
%
if $(D_1,E_1)<(D_2,E_2)$ are two cuts in $R$, then $E_1\cap
D_2\ne\emptyset$ and therefore, $D_1\pF\leq E_1\mF <D_2\pF\leq E_2\mF$.

If $B_{S_0}(a_0,R)\ne R$ is a ball in $R$, and if we take $S$ as defined
in Lemma~\ref{cutfill}, then by our definition,
\[
\tilde\iota(B_{S_0} (a_0,R)\mR)=B_S(a_0,R)\mF\;\mbox{ \ \ and \ \ }\;
\tilde\iota(B_{S_0}(a_0,R)\pR)= B_S(a_0,R)\pF\>.
\]
This together with $\tilde \iota(R\mR)=R\mF$ and
$\tilde\iota(R\pR)=R\pF$ shows:

\begin{lemma}                               \label{etoe}
The embedding $\tilde\iota$ sends equivalent cuts to equivalent cuts.
Hence the preimage of a full set is full.
\end{lemma}

Let us also note:
\begin{proposition}
If $vR$ is cofinal in $vF$ (which implies that there is no $f\in F$ such
that $f>R$), then $\tilde\iota$ sends principal cuts to principal
cuts. Otherwise, no principal cut is sent to a principal cut.
\end{proposition}
\begin{proof}
A principal cut in $R$ is the upper or lower edge of a ball $B_\emptyset
(a_0,R)$. Take the ball $B_S(a_0,F)$ as in Lemma~\ref{cutfill}. By
definition, $\tilde{\iota}(B_\emptyset (a_0,R)\mR)=B_S(a_0,F)\mF$ and
$\tilde{\iota}(B_\emptyset (a_0,R)\pR)=B_S(a_0,F)\pF$. The latter cuts
are principal if and only if $S=\emptyset$. By Remark~\ref{remS}, $S=
\emptyset$ if and only if there is no $\gamma\in vF$ such that
$\gamma>vR$, that is, if and only if $vR$ is cofinal in $vF$.
\end{proof}

If there is at least one non-ball cut in $R$ that is filled in $F$, then
the embedding $\tilde\iota$ will {\it not} be continuous with respect to
the interval topology. Even worse:
\begin{proposition}
Take any extension $F|R$ of ordered fields. If there is at least one
non-ball cut in $R$ that is filled in $F$, then there exists no
embedding of $\cC(R)$ in $\cC(F)$ that is continuous with respect to the
interval topology and compatible with restriction.
\end{proposition}
\begin{proof}
Take $C=(D,E)$ to be a non-ball cut in $R$ that is filled in $F$. Then
Lemma~\ref{cutfill} shows that $\mbox{\rm Betw}_F (D,E)$ is equal to a
ball $B$ in $F$. In order to be compatible with restriction, an
embedding has to send $C$ to a cut $C'$ in $F$ which is equal to $B\pF$,
$B\mF$, or a proper cut in $B$. Suppose that $C'\ne B\pF$. Take any cut
$C_1< B\mF$ and consider the open interval $I=(C_1,B\pF)$ in $\cC(F)$.
Then the restriction of $I$ to $\cC(R)$ is an interval in $\cC(R)$ with last
element $C$. This shows that the preimage of $I$ under any embedding
compatible with restriction is not open, as follows from Lemma~\ref{oi}
since $C$ is not a principal cut.

In the case of $C'=B\pF$, choose $C_2\in \cC(F)$ such that $B\pF<C_2$
and consider the open interval $I=(B\mF,C_2)$ in $\cC(F)$. Its
restriction to $\cC(R)$ is an interval with first element $C$, hence
again not open.
\end{proof}

The problem is that an open interval in $\cC(F)$ can end in a set that
fills a cut from $R$, in which case its preimage in $\cC(R)$ will include
an endpoint. However, a full open set will have to enter the between set
from both sides, and so we obtain the following positive result if we
switch from the interval to the full topology:

\begin{proposition}                               \label{cont}
Assume that $vR$ is a convex subgroup of $vF$. Then the embeddings
$\tilde\iota:\>\cC(R)\rightarrow \cC(F)$ constructed above are exactly the
embeddings that are continuous with respect to the full topology and
compatible with restriction.
\end{proposition}
\begin{proof}
Take an embedding $\tilde\iota:\>\cC(R)\rightarrow \cC(F)$ as constructed
above. In view of Lemma~\ref{rescut}, $\tilde\iota$ is compatible with
restriction.

By virtue of Lemma~\ref{etoe}, in order to show that $\tilde\iota$ is
continuous with respect to the full topology, it suffices to show that
the preimage of any full open set $U$ is open in the interval topology
of $\cC(R)$. Take $C\in \cC(R)$ with $\tilde{\iota}(C)\in U$. Since $U$
is open in the interval topology of $\cC(F)$, there is an open interval
$I \subseteq U$ which contains $\tilde{\iota}(C)$. The preimage of $I$
under $\tilde{\iota}$ is again an interval, and if $C$ is not an
endpoint of it, then $C$ lies in some open subinterval of this preimage.

Now suppose that $C$ is an endpoint of the preimage of $I$. Then either
all cuts in $I$ on the left side of $\tilde{\iota}(C)$ restrict to $C$,
or all cuts in $I$ on the right side of $\tilde{\iota}(C)$ restrict to
$C$. In both cases, we have that more than one cut in $F$ restricts to
$C$. Since we have assumed $vR$ to be a convex subgroup of $vF$,
Lemma~\ref{rescut} shows that we are in one of the following cases:
\sn
a) $C$ is a non-ball cut,
\sn
b) $C$ is a principal cut,
\sn
c) $C=R\mR$ or $C=R\pR$.
\sn
In all three cases, by our construction of $\tilde{\iota}$, we have that
$\tilde{\iota}(C)=B\mF$ or $\tilde{\iota}(C)=B\pF$ for some ball $B$ in
$F$. Denote the restriction of $B\mF$ to $R$ by $C_1\,$, and the
restriction of $B\pF$ to $R$ by $C_2\,$. Then $C=C_1$ or $C=C_2\,$.

Since $U$ is assumed to be full, $B\mF,B\pF\in U$ and since $U$ is open,
$B\mF\in I_1$ and $B\pF\in I_2$ for some open intervals $I_1$ and $I_2$
contained in $U$.

We first deal with cases a) and b). In both cases, $B\mF$ is the
smallest cut that reduces to $C_1$ and $B\pF$ is the largest cut that
reduces to $C_2\,$. The open interval $I_1$ contains a cut on the left
of $B\mF$, which consequently restricts to a cut $C'_1<C_1\,$.
Similarly, $I_2$ contains a cut on the right of $B\pF$, which
consequently restricts to a cut $C'_2>C_2$. For every $C'\in (C_2,C'_2)$
we have that $\tilde{\iota}(C_2)<\tilde{\iota}(C')<\tilde{\iota}(C'_2)$,
hence $\tilde{\iota}(C') \in I_2$. This shows that $[C_2,C'_2)$ is
contained in the preimage of $I_2\,$. Similarly, it is shown that
$(C'_1,C_1]$ is contained in the preimage of $I_1\,$.

In case a), both $B\mF$ and $B\pF$ restrict to $C$, so we have $C=C_1=
C_2$. In case b), where $C=a\mR$ or $C=a\pR$ for some $a\in R$, $B\mF$
restricts to $a\mR$ and $B\pF$ restricts to $a\pR$. In both cases,
$(C'_1,C_1]\cup [C_2,C'_2)=(C'_1,C'_2)$. It follows that $C$ has the
open neighborhood $(C'_1,C'_2)$ which is contained in the preimage of $U$.

Now we consider case c). In this case, $\tilde{\iota}(C)=R\mF$, the
largest cut that restricts to $C_1=R\mR$, or $\tilde{\iota}(C)=R\pF$,
the smallest cut that restricts to $C_2=R\pR$. The open interval $I_1$
contains a cut on the right of $R\mF$, which consequently restricts to a
cut $C'_1>R\mR\,$. Similarly, $I_2$ contains a cut on the left of
$R\pF$, which consequently restricts to a cut $C'_2<R\pR$. For every
$C'\in (C'_2,R\pR)$ we have that $\tilde{\iota}(C'_2)<\tilde{\iota}(C')<
\tilde{\iota}(R\pR)$, hence $\tilde{\iota}(C') \in I_2$. This shows that
$(C'_2,R\pR]$ is contained in the preimage of $I_2\,$. Similarly, it is
shown that $[R\mR,C'_1)$ is contained in the preimage of $I_1\,$. Now
one of these two intervals is an open neighborhood of $C$.

It follows in all three cases that $C$ has an open neighborhood
which is contained in the preimage of $U$. This proves that the
restriction of $U$ is open.
%
%

\parm
Now assume that $\tilde{\iota}'$ is an embedding of $\cC(R)$ in $\cC(F)$,
compatible with restriction. Suppose that there is a cut $C$ in $\cC(R)$
such that its image $\tilde{\iota}'(C)$ is not in accordance with our
above construction.

First, we consider the case of $C=(D,E)$ being a non-ball cut. Then our
assumption and the compatibility with restriction yield that $D\pF<
\tilde{\iota}'(C)<E\mF$ in $\cC(F)$. If the ball $B_S(a,F)$ is chosen
as in Lemma~\ref{cutfill}, then $D\pF=B_S(a,F)\mF$ and $E\mF=
B_S(a,F)\pF$. Therefore, the open interval $(D\pF,E\mF)$ in $\cC(F)$ is
full by Lemma~\ref{fullballs}. But the preimage of this interval is the
singleton $\{C\}$, hence not open.

Now we consider the case of $C=B_0\pR$ for some ball $B_0$ in $R$; the
case of $C=B_0\mR$ is symmetrical. If $(D,E)$ is the ball complement of
$B_0$ in $R$, then our assumption and the compatibility with restriction
yield that $D\pF<B_0\pF\leq\tilde{\iota}'(C)<E\mF$ in $\cC(F)$. The same
argument as before shows that $(D\pF,E\mF)$ is a full open interval in
$\cC(F)$. Its preimage in $\cC(R)$ has $C$ as its last element. Since
$C$ is an upper edge of a ball not equal to $R$, it follows that this
interval is not open.

Finally, we consider the case of $C=R\pR$; the case of $C=R\mR$ is
symmetrical. Then our assumption and the compatibility with restriction
yield that $R\pF<\tilde{\iota}'(C)$ in $\cC(F)$. The open set
$[C_{-\infty},R\mF)\cup (R\pF,C_{\infty}]$ in $\cC(F)$ is full by
Lemma~\ref{fullballs}. But the preimage of it is either $\{R\pR\}$ or
$\{R\mR,R\pR\}$, hence not open.
\end{proof}

Our positive result is contrasted by the following negative result:

\begin{proposition}
Assume that $vR$ is not a convex subgroup of $vF$. Then there are no
embeddings $\tilde\iota:\>\cC(R)\rightarrow \cC(F)$ that are continuous
with respect to the full topology and compatible with restriction.
\end{proposition}
\begin{proof}
If $vR$ is not a convex subgroup of $vF$, then there are $\alpha,\beta
\in vR$ and $\gamma\in vF\setminus vR$ such that $\alpha<\gamma<\beta$.
Take $S_0:=\{\delta\in vR\mid \gamma<\delta\}$ and $B_0:=B_{S_0}(0,R)$.
Note that $B_0\ne R$ because $\alpha\notin S_0\,$, and that $B_0$ is not
a singleton because $\beta\in S_0\,$.

Now if $B_S(0,F)$ is as in Lemma~\ref{cutfill}, then it follows from
Remark~\ref{remS} that $\gamma\in S\setminus S_0\,$. This implies that
$B_{S_0}(0,R)$ is not cofinal in $B_S(0,F)$, whence $B_0\pF<
B_S(0,F)\pF$. Now assume that $\tilde\iota:\>\cC(R)\rightarrow \cC(F)$
is an embedding compatible with restriction. Then by Lemma~\ref{rescut},
$B_0\pF\leq \tilde\iota(B_0\pR)\leq B_S(0,F)\pF$. Suppose first that
$B_0\pF< \tilde\iota(B_0\pR)$. By Lemma~\ref{fullballs}, the open
neighborhood $U:= [C_{-\infty},B_0\mF) \cup (B_0\pF,C_{\infty}]$ of
$\tilde\iota (B_0\pR)$ in $\cC(F)$ is full. But $\tilde\iota^{-1}(U)=
[C_{-\infty}, B_0\mR) \cup [B_0\pR,C_{\infty}]$ or $\tilde\iota^{-1}(U)=
[C_{-\infty},B_0\mR] \cup [B_0\pR,C_{\infty}]$ in $\cC(R)$, both of
which are not open since $B_0$ is not a singleton and therefore $B_0\pR$
is not the immediate successor of $B_0\mR$.

Suppose now that $B_0\pF=\tilde\iota(B_0\pR)$. Again by
Lemma~\ref{fullballs}, the open neighborhood $U:=(B_S(0,F)\mF,
B_S(0,F)\pF)$ of $\tilde\iota (B_0\pR)$ in $\cC(F)$ is full. But
$\tilde\iota^{-1}(U)= (B_0\mR,B_0\pR]$ or $\tilde\iota^{-1}(U)=
[B_0\mR,B_0\pR]$ in $\cC(R)$, both of which are not open since
$B_0\ne R$.
\end{proof}

%
%
%
%
\section{Embeddings of $M(R(y))$ in $M(F(y))$}    \label{embM}
We will now consider an extension of formally real fields $F|R$, with
$R$ real closed, but not necessarily archimedean. We will first consider
the case where also $F$ is real closed.

We assume that $vR$ is convex in $vF$ and start from one of the
embeddings $\tilde\iota:\>\cC(R)\rightarrow \cC(F)$ constructed in the
previous section (cf.\ Proposition~\ref{cont}). We define an embedding
\[
\iota:\>M(R(y))\longrightarrow M(F(y))
\]
in the following way. If $M(R(y))\ni\zeta=\lambda_{R(y)}\circ \chi_R(C)$
for a cut $C$ in $R$, then we set
\[
\iota(\zeta)\>:=\> \lambda_{F(y)}\circ\chi_F(\tilde\iota(C))\;.
\]
Since $\tilde\iota$ is compatible with the equivalence of cuts, the
embedding $\iota$ is well-defined and the diagram

\begin{displaymath}
{
\begin{array}{ccc}
\cC(F) & \stackrel {\lambda_{F(y)}\circ \chi_F^{ }}{\longrightarrow }
& M(F(y))\\\\
{\tilde\iota}{ \Big {\uparrow}} & &{\iota}{\Big  \uparrow} \\\\
\cC(R) &\stackrel {\lambda_{R(y)}\circ \chi_R^{ }}{\longrightarrow }
& M(R(y))
\end {array}
}
\end{displaymath}
%
%
%
commutes.

\begin{theorem}                             \label{thembM}
Take an extension $F|R$ of real closed fields. If $vR$ is convex in
$vF$, then the embedding $\iota$ as defined above does not depend on the
particular choice of $\tilde{\iota}$ and is continuous and compatible
with restriction.

Conversely, if $\iota:M(R(y))\rightarrow M(F(y))$ is continuous and
compatible with restriction, then it induces an embedding
$\tilde{\iota}:\cC(R)\rightarrow\cC(F)$ continuous w.r.t.\ the full
topology and compatible with restriction, such that the above
diagram commutes, and $vR$ is convex in $vF$.
\end{theorem}
\begin{proof}
Take $\tilde\iota$ as constructed in the previous section. We show that
$\iota$ is continuous. Take any open set $U$ in $M(F(y))$. By
Proposition~\ref{lc}, its preimage $U_1$ in $\cC(F)$ is a full open set.
Then by Proposition~\ref{cont}, the preimage $U_2$ of $U$ in $\cC(R)$ is
a full open set. Again by Proposition~\ref{lc}, the image $U_3$ of $U_2$
in $M(R(y))$ is open.
%
%
From Lemma~\ref{ipri} we know that $\res(U)$ is the preimage of $U$
under $\iota$. But from the commutativity of the diagram in
Proposition~\ref{diag} we know that
\[
\res(U)\>=\>\res\circ\lambda_{F(y)}\circ \chi_F(U_1)\>=\>
\lambda_{R(y)} \circ \chi_R\circ\res (U_1)\>=\>U_3\;.
\]
So the preimage of $U$ under $\iota$ is open. This proves the continuity
of $\iota$.

\pars
In the construction of $\tilde\iota$ in the previous section the only
freedom we had was to choose either the upper or the lower edge of the
ball which fills a non-ball cut in $R$; but these cuts correspond to the
same $\mathbb R$-place in $M(F(y))$. This shows that all embeddings
$\tilde \iota$ constructed in the previous section determine the same
embedding $\iota$.

\pars
We will now prove the second assertion. Take $\iota$ as in the
assumption. For each $C\in\cC(R)$, we wish to define $\tilde\iota(C)$
such that
\[
\lambda_{F(y)}\circ \chi_F\circ\tilde\iota\,(C)\>=\>
\iota\circ\lambda_{R(y)} \circ \chi_R\,(C)\;.
\]
Set $\xi:= \lambda_{R(y)} \circ \chi_R(C)\in M(R(y))$ and $\xi':=
\iota(\xi)$. Since $\iota$ is compatible with restriction, $\xi$ is the
restriction of $\xi'$ to $R(y)$. By the commutativity of the diagram in
Proposition~\ref{diag}, we find that if $C'\in\cC(F)$ is sent to $\xi'$
by $\lambda_{F(y)}\circ \chi_F\,$, then $\res(C')$ must be sent to $\xi$
by $\lambda_{R(y)} \circ \chi_R$.

If $C$ is a non-ball cut, then choose any $C'\in\cC(F)$ such that
$\lambda_{F(y)}\circ \chi_F(C'_1)=\xi'$ and define $\tilde\iota(C):=
C'$. Since $C$ is the only cut in $R$ that is sent to $\xi$ by
$\lambda_{R(y)} \circ \chi_R$, it follows that $\res(C')=C$.

If $C$ is a ball cut, that is, $C=B_0\mR$ or $C=B_0\pR$ for some ball
$B_0$ in $R$, then we have to find images for both $B_0\mR$ and
$B_0\pR$. We claim that the continuity of $\iota$ implies that the
preimage of $\xi'$ under $\lambda_{F(y)}\circ \chi_F$ is $\{B\mF,B\pF\}$
for some ball $B$ in $F$ with $\res(B\mF)=B_0\mR$ and $\res(B\pF)=
B_0\pR$. We treat the case of $B_0\ne R$ and leave the case of $B_0=R$
to the reader.

We write $B_0=B_{S_0}(a_0,R)$, take $S$ as in Lemma~\ref{cutfill}, and
set $B:=B_S(a_0,F)$. Suppose the preimage of $\xi'$ is not $\{B\mF,
B\pF\}$. Take $C'$ in the preimage. Then by what we have shown above,
$C'$ restricts to $B_0\mR$ or $B_0\pR$. We assume the latter case; the
former is symmetrical. Then $B_0\pF\leq C'<B\pF$. By
Proposition~\ref{fullballs}, the open interval $(B\mF, B\pF)$ is full,
so $U:=\lambda_{F(y)}\circ \chi_F ((B\mF, B\pF))$ is open in $M(F(y))$
and contains $\xi'$. The restriction $I$ of $(B\mF, B\pF)$ to $\cC(R)$
has $B_0\pR= \res(C')$ as its largest element, hence it is not open. The
same argument as in the first part of this proof shows that the preimage
$U'$ of $U$ under $\iota$ is equal to $\lambda_{F(y)}\circ \chi_F\circ
\res\, ((B\mF, B\pF))=\lambda_{F(y)}\circ \chi_F(I)$, which is not open.
But this contradicts the continuity of $\iota$. We see that the preimage
of $\xi'$ must be $\{B\mF, B\pF\}$. So we set $\tilde{\iota} (B_0\mR)=
B\mF$ and $\tilde{\iota}(B_0\pR)=B\pF$ and note that $\res(B\mF)=B_0\mR$
and $\res(B\pF)=B_0\pR$.

We have now defined a mapping $\tilde\iota: \cC(R)\rightarrow \cC(F)$
which is compatible with the restriction. Therefore, $\tilde\iota$
must be injective, and since the restriction preserves $\leq$ by
Lemma~\ref{contresC}, $\tilde\iota$ must preserve $<\,$. By definition,
$\tilde\iota$ also preserves equivalence.

It remains to show that $\tilde\iota$ is continuous w.r.t.\ the full
topology. Take a full open set $U$ in $\cC(F)$. By Proposition~\ref{lc},
$U_1:=\lambda_{F(y)}\circ\chi_F(U)$ is open. By Lemma~\ref{ipri}, $U_2
:=\res(U_1)$ is the preimage of $U_1$ under $\iota$, hence open since
$\iota$ is continuous. By the commutativity of the diagram in
Proposition~\ref{diag},
\[
U_2\>=\>\res\circ\lambda_{F(y)}\circ \chi_F(U)\>=\>
\lambda_{R(y)} \circ \chi_R\circ\res\, (U)\;.
\]
Thus, the full set $\res(U)$ in $\cC(R)$ is the preimage of $U_2\,$,
hence open by Proposition~\ref{lc}. Again by Lemma~\ref{ipri},
the full open set $\res(U)$ is the preimage of $U$ under $\tilde\iota$.
This proves the continuity of $\tilde\iota$.
\end{proof}

\parm
Now we will consider the case of $F$ not being real closed. We choose a
real closure $R'$ of $F$ and take $\iota':\>M(R(y))\rightarrow M(R'(y))$
to be the embedding constructed above. Since $\res_{R'(y)|F(y)}$ is
continuous (cf.\ Remark~\ref{full}, part 3)\,)
\[
\iota\>:=\>\res_{R'(y)|F(y)}\circ \iota'
\]
is a continuous mapping from $M(R(y))$ to $M(F(y))$. Since $\iota'$ is
compatible with the restriction
\[
\res_{R'(y)|R(y)}\>=\>\res_{F(y)|R(y)}\circ\res_{R'(y)|F(y)}\;,
\]
we see that $\iota$ is compatible with the restriction. For this reason,
it is also injective.

As the real closure $R'$ can be taken with respect to any ordering on
$F$, we may lose the uniqueness of $\iota$, However, we are able to show
the following partial uniqueness result:

\begin{theorem}                             \label{uniqnotrc}
Take two orderings $P_1$ and $P_2$ of $F$ which induce the same
$\R$-place, $R'_1$ and $R'_2$ the respective real closures of $F$, and
$\iota'_i:\>M(R(y))\rightarrow M(R'_i(y))$, $i=1,2$, the unique
continuous embeddings compatible with restriction. Consider the
following commuting diagram:
\begin{displaymath}
\xymatrix{
  & M(R'_1(y)) \ar[rd]|{ \text{res}_1} &\\
M(R(y)) \ar[ru]|{ \iota'_1} \ar[rd]|{ \iota'_2}
&&M(F(y))\ar[ll]|{ \text{res}} \\
 & M(R'_2(y)) \ar[ru]|{ \text{res}_2} & }
\end{displaymath}
Then
\[
\res_1\circ \iota'_1\>=\>\res_2\circ \iota'_2\>.
\]
\end{theorem}
\begin{proof}
We will first show that the mappings coincide on all $\mathbb R$-places
of $R(y)$ determined by the principal cuts.

Suppose that $\zeta=\lambda\circ\chi(a^+)=\lambda\circ\chi(a^-)$, where
$a\in R$. Note that for the corresponding valuation $v_\zeta$ on $R(y)$,
we have that $vR< v_\zeta(a-y)$. Let $\zeta_i:=\iota'_i(\zeta)$, for
$i=1,2$. By the definition of the embedding $\iota'_i$, we have that
$\zeta_i$ is determined by the upper and lower edge of the ball
$B_{S_i}(a,R'_i)$ where $S_i=\{\alpha\in vR'_i\mid \alpha>vR\}$. Then
for the corresponding valuation $v_{\zeta_i}$ on $R'_i(y)$ we have that
$vR<v_{\zeta_i}(a-y)<S_i$ in $v_{\zeta_i}R'_i$. Since these value groups
are divisible (by \cite[Theorem 4.3.7]{ep}, $R'_i$ being real closed
fields), the values $v_{\zeta_i}
(a-y)$ are rationally independent over these value groups. Therefore,
the valuations $v_{\zeta_i}$ are uniquely determined by the natural
valuations on $R'_i$ and the values $v_{\zeta_i}(a-y)$. The same remains
true when we restrict to $F(y)$. There, by our assumption, the
restrictions of the natural valuations on $R'_i$ coincide, so the
restrictions of the valuations $v_{\zeta_i}$ to $F(y)$ must coincide,
too. Further, the residue fields of $v_{\zeta_i}$ on $F(y)$ are equal to
the residue field of $F$ because $v_{\zeta_i}(a-y)$ is rationally
independent over $vF$. Since the restrictions to $F$ of ${\zeta_1}$ and
${\zeta_2}$ coincide, the restrictions to $F(y)$ of these $\mathbb
R$-places coincide, as well. Therefore, $\res_1\circ \iota'_1(\zeta)
\>=\>\res_2\circ \iota'_2(\zeta)\>$.

Now take $\zeta_1=\res_1\circ \iota'_1(\zeta)$ and $\zeta_2 =\res_2
\circ \iota'_2(\zeta)$ for some $\zeta\in M(R(y))$ and suppose they are
distinct. Since $M(F(y))$ is Hausdorff, there are disjoint open
neighborhoods $U_1\ni \zeta_1$ and $U_2\ni \zeta_2\,$. The preimages of
$U_1$ and $U_2$ in $M(R(y))$ are open, and $\zeta$ lies in their
intersection. So this intersection is not empty, and by the density of
the principal places in $M(R(y))$ (cf.\ Lemma~\ref{ppdens}), there is a
principal place $\zeta_0$ in this intersection. But the images of
$\zeta_0$ under the two embeddings are equal and hence must lie in
$U_1\cap U_2\,$, a contradiction.
\end{proof}

%
%
%
%
\section{Embeddings of $M(R(y))$ in $M(F(y))$ for
archimedean $R$}                                \label{embMrc}
In this section we will consider an extension of formally real fields
$F|R$ in the special case where $R$ is archimedean real closed. The
general case has been treated in the previous section. Here, we wish to
give a different, more explicit construction of a continuous
embedding $\iota$ of $M(R(y))$ in $M(F(y))$ which is compatible with
restriction.

We choose any real place $\xi$ of $F$. Then $\ovl{F}:=\xi(F)\subseteq
\R$. Since $R$ is archimedean, we can assume that $\xi|_R=\mbox{\rm
id}_R\,$ and that $\ovl{F}|R$ is an extension of archimedean ordered
fields. By $\xi_y$ we denote the \bfind{constant extension} of $\xi$ to
$F(y)$, i.e., the unique extension of $\xi$ which is trivial on $R(y)$.
Its valuation ring is the smallest subring of $F(y)$ containing both the
valuation ring of $\xi$ and $R(y)$. The valuation associated with
$\xi_y$ is the {\bf Gau{\ss}} or {\bf functorial valuation} on $F(y)$
extending the valuation associated with $\xi$ on $F$. On polynomials in
$F[y]$ with coefficients in the valuation ring of $\xi$, $\xi_y$ acts by
applying $\xi$ to the coefficients. Therefore, the residue field of
$\xi_y$ is $\xi(F)(y)$.

For every $\zeta\in M(R(y))$ we define the constant extension
$\zeta_{\ovl{F}}$ of $\zeta$ to $\ovl{F}(y)$ as follows. As $\zeta$ is
trivial on the archimedean field $R$, it is determined by an irreducible
polynomial $p(y)\in R[y]$ (or by $1/y$). Since $R$ is real closed and
$\ovl{F}$ is formally real, such a polynomial $p$ remains irreducible
over $\ovl{F}$ and thus, $p$ (or $1/y$, respectively) determines a
unique extension of $\zeta$ to $\ovl{F}(y)$ which is trivial on
$\ovl{F}$. We set $\iota_{\ovl{F}|R}(\zeta):=\zeta_{\ovl{F}}\,$.

\begin{lemma}                               \label{cec}
The mapping $\iota_{\ovl{F}|R}:\>M(R(y))\,\rightarrow\,M(\ovl{F}(y))$ is
a continous embedding compatible with the restriction. If $\ovl{F}$ is
real closed, then it is a homeomorphism.
\end{lemma}
\begin{proof}
Since $\ovl{F}|R$ is an extension of archimedean ordered fields, $R$
lies dense in $\ovl{F}$. It follows from \cite[Theorem~3.2]{kmo}
that the restriction mapping from $M(\ovl{F}(y))$ to $M(R(y))$ is a
homeomorphism if $\ovl{F}$ is real closed. Hence in this case,
$\iota_{\ovl{F}|R}$ is a homeomorphism.

If $\ovl{F}$ is not real closed, then we consider a real closure $R'$ of
$\ovl{F}$. By what we have shown already, $\iota_{R'|R}$ is a
homeomorphism. Since $\res_{R'(y)|R(y)}$ is continuous, the same holds
for $\iota_{\ovl{F}|R}=\res_{R'(y)|\ovl{F}(y)}\circ \iota_{R'|R}$.
\end{proof}

Now we define
\begin{equation}                            \label{emb}
\iota(\zeta)\>:=\>\zeta_{\ovl{F}}\circ\xi_y\>.
\end{equation}

\begin{theorem}
The mapping $\iota:\>M(R(y))\,\rightarrow\,M(F(y))$ is a continous
embedding.
\end{theorem}
\begin{proof}
Take $a\in F(y)$. We have to show that the preimage of a subbasis set
$H'(a)$ under $\iota$ is open in $M(R(y))$. If $\xi_y(a)$ is 0 or
$\infty$, then the same holds for $\zeta_{\ovl{F}}\circ\xi_y$ for every
$\zeta\in M(R(y))$. In this case, $H'(a)$ is empty and we are done.

Assume now that $\xi_y(a)\ne 0,\infty$. Then $\xi_y(a)$ is a nonzero
rational function $g(y)\in \ovl{F}(y)$. The preimage of $H'(a)$ is then
the set of all real places $\zeta\in M(R(y))$ such that $\zeta_{\ovl{F}}
(g)>0$. In the case of $\ovl{F}=R$ (which for instance holds when
$R=\R$), this is precisely $H'(g)$ in $M(R(y))$. For the general case,
we apply Lemma~\ref{cec} to conclude that the preimage of $H'(g)$ under
the constant extension mapping $\zeta\mapsto\zeta_{\ovl{F}}\,$, and
hence the preimage of $H'(a)$ under $\iota$, is open.
\end{proof}

From Theorem~\ref{uniqnotrc}, where we take $F=R(x)$, we now obtain:

\begin{theorem}                              \label{thembMrc}
The mapping $\iota$ defined in (\ref{emb}) is the unique continuous
embedding of $M(R(y))$ in $M(R(x,y))$ that is compatible with
restriction and such that all places in the image of $\iota$ have
restriction $\xi$ to $R(x)$.
\end{theorem}

We have chosen to give a direct proof of Theorem~\ref{thembMrc} although
it can be derived from the theorems of the last section. In order to do
this, we have to show that the embedding defined in (\ref{emb})
coincides with the embedding we have constructed before. To this end, we
consider an ordering $P$ of $R(y)$ and the cut $C$ it induces in the
archimedean real closed field $R$. If $R=\R$, then the only
possibilities are $C= C_{\infty}$, $C= C_{-\infty}$, or $C=r^+,r^-$ for
$r\in\R$. If $R\ne\R$, $C$ can also be a cut induced in $R$ by some real
number $r\in\R \setminus R$.

If $C= C_{\infty}$ or $C= C_{-\infty}$, we have that $y>F$ or $y<F$
under the corresponding orderings. In this case, $0<vy^{-1}<vF^{>0}$,
where $vF^{>0}$ denotes the set of positive elements of $vF$.

In the case of $C=r^+,r^-$, we have that $\iota(C)$ is the upper or
lower edge of $B_{vF^{>0}}(r,F)$. This ball is $r+{\cal M}$ where
${\cal M}$ is the valuation ideal of infinitesimals in $F$. Since $C$ is
induced by $y$, we find that $0<v(y-r)<vF^{>0}$.

In the final case, we have two subcases. If $C$ is not filled in $F$,
then $v(y-f)\leq 0$ for every $f\in F$. If $C$ is filled by some
element in $F$, then we can identify this element with the real number
$r$ that fills the cut $C$. In this case, we obtain the same result as
in the previous case.

In all three cases, we find the constant extension $\xi_y$ of $\xi$ must
be trivial on $R(y)$, which implies that $\iota(\zeta)$ must be of the
form $\zeta_{\ovl{F}}\circ\xi_y$.

\pars
In the case of $R=\R$, we can show the above more directly:
\begin{proposition}
Take $\iota$ to be an embedding of $M(\R(y))$ in $M(\R(x,y))$,
compatible with restriction and such that all places in the image of
$\iota$ have the same restriction to $\R(x)$. If for some
$\xi\in\im(\iota)$ such that $\xi(x)=a$ and $\xi(y)=b$ we have that for
some $n\in\N$,
\[
0\><\>v_{\xi}(x-a)\>< n v_{\xi}(y-b)\>,
\]
then the embedding is not continuous. The same holds if $\xi(x)=\infty$
and $x-a$ is replaced by $1/x$ and/or $\xi(y)=\infty$ and $y-b$ is
replaced by $1/y$.
\end{proposition}
\begin{proof}
Take
\[
f(x,y)\>=\>\frac{x-a+(y-b)^n}{x-a}\>.
\]
Then $H'(f)\cap \im(\iota)$ is the singleton $\{\xi\}$. Indeed, $\xi\in
H'(f)$ since $\xi(f)=1$. But if $\xi'=\iota(\zeta)\ne\xi$, then
$\zeta(y)\ne b$, whence $\xi'(f)=\infty$. The cases of $\xi(x)=\infty$
and/or $\xi(y)=\infty$ are similar.
\end{proof}

It is possible to generalize the approach of this section to the
general setting of the previous section by replacing the $\R$-place
$\xi$ of $F$ by the finest coarsening $\xi'$ whose residue field
contains $R$. (The valuation ring of $\xi'$ is the compositum of the
valuation ring of $\xi$ and the subfield $R$ of $F$.) But we would need
an analogue of Lemma~\ref{cec} for the case of non-archimedean fields
$R$ and $\ovl{F}=\xi'(F)$. We found that the tools developed to deal
with this analogue can be directly applied to construct the embedding of
$M(R(y))$ in $M(F(y))$ in the setting of the previous section.

%
%
%
%
\section{Embeddings of $\prod_{i=1}^{n} M(\mathbb R(x_i))$ in
$M(\mathbb R(x_1,\dots,x_n))$}                        \label{embtor}
In order to study possible embeddings of the torus in spaces of real
places, we wish to consider embeddings of $M(\mathbb R(x))\times
M(\mathbb R(y))$ in $M(\mathbb R(x,y))$. Initially, we will treat the
more general case of $n$ variables. We consider the projection mapping
\[
\rho: M(\mathbb R(x_1,\dots,x_n))\ni\xi\>\mapsto\>
(\xi|_{\mathbb R(x_1)},\ldots,\xi|_{\mathbb R(x_n)})\in
\prod_{i=1}^{n} M(\mathbb R(x_i)) \>.
\]

\begin{lemma}
The mapping $\rho$ is surjective.
\end{lemma}
We describe a general construction that will prove the lemma. Take
$\R$-places $\xi_i\in M(\R(x_i))$. We wish to associate to them an
$\R$-place $\xi$ of $\R(x_1,\dots,x_n)$ whose restriction to $\R(x_i)$
is $\xi_i\,$. We may assume that $\xi_i (x_i)\ne\infty$; otherwise, we
can replace $x_i$ by $1/x_i$. For $1\leq i<n$, let $\xi'_i$ be the place
of $\R(x_i,\ldots,x_n)$ which is trivial on $\R(x_{i+1},\ldots, x_n)$
and such that $\xi'_i(x_i)=\xi_i (x_i)$. Its residue field is
$\R(x_{i+1},\ldots, x_n)$. Then the place
\begin{equation}                            \label{stack}
\xi\>=\>\xi_n\circ\xi'_{n-1}\circ\ldots\circ\xi'_1\>.
\end{equation}
satisfies the above conditions. This construction can be
replaced by the symmetric ones where the $x_i$ are permuted.

\begin{remark}                              \label{many}
There are many more possibilities for choosing a common extension $\xi$
of the $\xi_i\,$. Set $\xi_i (x_i)=a_i\,$. Choose any rationally
independent elements $r_1,\ldots,r_n\in\R$. Then there is a (uniquely
determined) $\R$-place $\xi$ of $\R(x_1,\dots,x_n)$ such that for the
valuation $v$ associated with $\xi$ we have that $v(x_i-a_i)=r_i\,$. The
value group of $\xi$ is generated by the values $r_1,\ldots,r_n$ and is
thus archimedean. In contrast to this, the value group of the place in
(\ref{stack}) has rank $n$ and is thus not archimedean if $n>1$.
\end{remark}

\pars
The surjectivity shows that there exist embeddings
\begin{equation}                            \label{embprod}
\iota:\> \prod_{i=1}^{n} M(\mathbb R(x_i)) \>\hookrightarrow\>
M(\mathbb R(x_1,\dots,x_n))\>.
\end{equation}
Such an embedding will be called {\bf compatible} if $\rho\circ\iota$ is
the identity.

\begin{theorem}                             \label{cd}
The image of every compatible embedding $\iota$ as in (\ref{embprod})
lies dense in $M(\mathbb R(x_1,\dots,x_n))$. But for $n>1$, every
non-empty basic open subset of $M(\mathbb R(x_1,\dots,x_n))$ contains
infinitely many places that are not in the image of $\iota$.
\end{theorem}
\begin{proof}
Take non-zero elements $f_1,\ldots,f_m\in \R(x_1,\dots,x_n)$ such that
\[
U:= H'(f_1)\cap\ldots\cap H'(f_m)\>\ne\>\emptyset\>.
\]
Take $\zeta\in U$ and write $f_i(x_1,\dots,x_n)=
\frac{g_i(x_1,\dots,x_n)}{h_i(x_1,\dots,x_n)}$. Choose an ordering on
$\R(x_1,\dots,x_n)$ compatible with $\zeta$. Then the existential
sentence
\[
\exists X_1\ldots \exists X_n :\; \bigwedge_{1\leq i\leq m}
h_i(X_1,\dots,X_n)\ne 0 \,\wedge\,
\frac{g_i(X_1,\dots,X_n)}{h_i(X_1,\dots,X_n)} >0
\]
holds in $\R(x_1,\dots,x_n)$ with this ordering. By Tarski's Transfer
Principle, it also holds in $\R$ with the usual ordering. That is, there
exist $a_1,\dots,a_n\in\R$ such that $h_i(a_1,\dots,a_n)\ne 0$ and
$\frac{g_i(a_1,\dots,a_n)}{h_i(a_1,\dots,a_n)} >0$ for $1\leq i\leq m$.
Hence for every $\R$-place $\zeta\in M(\R(x_1,\dots,x_n))$ such that
$\zeta (x_i)=a_i$ we will have that $\zeta (f_i)=
\frac{g_i(a_1,\dots,a_n)}{h_i(a_1,\dots,a_n)} >0$. Among all such
$\zeta$ there is precisely one in $\im(\iota)$. For this $\zeta$, we
have that $\zeta\in U\cap\im (\iota)$. This proves that $\im (\iota)$
lies dense in $M(\R(x_1,\dots,x_n))$.

For $n>1$, Remark~\ref{many} shows that there are infinitely many
$\R$-places $\zeta\in M(\R(x_1,\dots,x_n))$ such that $\zeta (x_i)=
a_i\,$. As only one of them is in $\im(\iota)$, $U\setminus\im(\iota)$
is infinite.
\end{proof}

\begin{corollary}
A compatible embedding $\iota$ as in (\ref{embprod}) cannot be
continuous with respect to the product topology on $\prod_{i=1}^{n}
M(\mathbb R(x_i))$.
\end{corollary}
\begin{proof}
Suppose we have a continuous compatible embedding. Under the product
topology, the space $\prod_{i=1}^{n} M(\mathbb R(x_i))$ is compact. As
the continuous image of a compact space in a Hausdorff space is again
compact (cf.\ \cite{K}, Chapter 5, Theorem 8), we find that the image
is closed in $M(\mathbb R(x_1,\dots,x_n))$. As it is also dense in
$M(\mathbb R(x_1,\dots,x_n))$ by Theorem~\ref{cd}, it must be equal to
$M(\mathbb R(x_1,\dots,x_n))$. But this contradicts the second assertion
of Theorem~\ref{cd}. Hence the embedding cannot be continuous.
\end{proof}

\begin{remark}
All of the above can be generalized to the case of infinitely many
elements $x_i\,$, $i\in I$ that are algebraically independent over
$\R$. After choosing some well-ordering on $I$, the construction of the
embedding
\[
\iota:\> \prod_{i\in I} M(\mathbb R(x_i)) \>\hookrightarrow\>
M(\mathbb R(x_i\mid i\in I))
\]
proceeds by (possibly transfinite) induction. The above theorem and
corollary remain valid. The proof of the theorem still works, as in the
finitely many polynomials $f_1,\ldots,f_m$ only finitely many variables
$x_i$ can appear. For infinite $I$, it is no longer true that the choice
of the elements $a_1,\dots,a_n$ determines a unique place in
$\im(\iota)$. Still, an application of Remark~\ref{many} shows that
$U\setminus\im(\iota)$ is infinite.
\end{remark}

\pars
We will now reprove the result of the corollary in the case of $n=2$ by
looking more closely at the topologies that are involved here. Every
embedding of $M(\mathbb R(x))\times M(\mathbb R(y))$ in $M(\mathbb
R(x,y))$ will induce a topology on $M(\mathbb R(x)) \times M(\mathbb
R(y))$ whose open sets are the preimages of the intersections of the
open sets of $M(\mathbb R(x,y))$ with the image of the embedding.

\begin{theorem}
For every compatible embedding $\iota$, the topology induced on
\linebreak
$M(\mathbb R(x))\times M(\mathbb R(y))$ is finer than the product
topology.
\end{theorem}
\begin{proof}
Take a basic open set in the product topology of $M(\mathbb R(x))\times
M(\mathbb R(y))$ which is the interior or exterior of a circle given by
$(x-a)^2+(y-b)^2=r^2$, where $a,b,r\in\R$. We set
\[
f(x,y)\>=\>r^2-(x-a)^2-(y-b)^2\>.
\]
Then the set $\im(\iota)\cap H'(f)$ is precisely the image of the
interior of the circle, and set $\im(\iota)\cap H'(-f)$ is precisely the
image of the exterior of the circle. This proves that the induced
topology is equal to or finer than the product topology.

It remains to present an induced open set in $M(\mathbb R(x))\times
M(\mathbb R(y))$ which is not open in the product topology. Take
the unique $\xi$ in $\im(\iota)$ such that $\xi(x)=0$ and $\xi(y)=0$.
\[
f(x,y)\>=\>\left\{
\begin{array}{ll}
1+\frac{x}{y} & \mbox{if } \xi(\frac{x}{y})=0  \\
1+\frac{y}{x} & \mbox{if } \xi(\frac{y}{x})=0  \\
\frac{y^2}{x^2} & \mbox{otherwise.}
\end{array}  \right.
\]
It follows in all three cases that $\xi\in H'(f)$. The preimage of $\xi$
under $\iota$ is $(\xi_1,\xi_2)$ where $\xi_1(x)=0$ and $\xi_2(y)=0$. If
the subset $U$ induced by $H'(f)$ in $M(\mathbb R(x))\times M(\R(y))$
would be open, then it would contain the interior of a circle $x^2+y^2
=r^2$ for some $r>0$. But this is impossible since whenever $(\xi_1,
\xi_2)\in U$, then for the first choice of $f$, $\xi_2(y)=0$ must imply
$\xi_1(x)=0$, and for the two other choices of $f$, $\xi_1(x)=0$ must
imply $\xi_2(y)=0$.
\end{proof}

\sn
{\bf Open Problem:} What is the induced topology? Is it one-dimensional
or two-dimensional?


%
%
%
%
\section{Embeddings of more general products}
For simplicity, we will only consider the product of two spaces $M(F_1)$
and $M(F_2)$; a generalization to any finite products can be achieved
along the lines of the last section. We will also assume that $F_1$ and
$F_2$ both contain $\R$. Then we can assume them embedded in some
extension field of $\R$ such that $F_1$ and $F_2$ are linearly disjoint
over $\R$. We denote by $F$ the field compositum of $F_1$ and $F_2$,
that is, the smallest subextension of the given extension of $\R$ that
contains both $F_1$ and $F_2$.

As before, we consider the corresponding projection mapping
\[
\rho:\> M(F)\ni\xi \>\mapsto\> (\xi|_{F_1},\xi|_{F_2})\in M(F_1)\times
M(F_2) \>.
\]
We show that $\rho$ is surjective. Take $(\xi_1,\xi_2)\in M(F_1)\times
M(F_2)$. Then there is an extension $\xi'_1$ of $\xi_1$ from $F_1$ to
$F$ such that the residue field of $\xi'_1$ is $F_2$. Then take
$\iota(\xi_1,\xi_2) = \xi_2\circ\xi'_1$. Here again, one obtains a
different place of $F$ by interchanging $F_1$ and $F_2\,$, showing that
$\rho$ is not injective.

\pars
The surjectivity shows that there exist embeddings
\[
\iota:\>M(F_1)\times M(F_2)\>\longrightarrow\>M(F)
\]
As before, $\iota$ will be called {\bf compatible} if $\rho\circ\iota$
is the identity.

If $F_1|\R$ and $F_2|\R$ are function fields, we can again prove that
the image of every compatible embedding $\iota$ lies dense in $M(F)$. We
will need the following fact. For a proof, see the second half of the
proof of the lemma on p.~190 of \cite{kp}.

\begin{lemma}                               \label{KP}
Take a field $k$ and a function field $K=k(x_1,\ldots,x_d,z)$ where
$x_1,\ldots,x_d$ are algebraically independent over $k$ and $z$ is
separable-algebraic over $k(x_1,\ldots,x_d)$. If $f\in k[x_1,\ldots,
x_d, Z]$ is the irreducible polynomial of $z$ over $k(x_1,\ldots,x_d)$
and if $a_1,\ldots,a_d,b\in k$ are such that
\[
f(a_1,\ldots,a_d,b)\>=\>0\mbox{ \ \ and \ \ }
\frac{\displaystyle\partial f}{\displaystyle\partial Z}
(a_1,\ldots,a_d,b)\>\ne\>0\>,
\]
then $K$ admits a $k$-rational place $\xi$ such that
$\xi(x_i)=a_i$ for $1\leq i\leq d$, and $\xi(z)=b$.
\end{lemma}

\begin{theorem}                         
If $F_1|\R$ and $F_2|\R$ are function fields of transcendence degree
$\geq 1$, then the image of every compatible embedding $\iota$ lies
dense in $M(F)$. But every non-empty basic open subset of $M(F)$
contains infinitely many places that are not in the image of $\iota$.
\end{theorem}
\begin{proof}
We write $F_1=\R(x_1,\ldots,x_d,z_1)$ and $F_2=\R(x_{d+1},\ldots,x_{d+e}
,z_2)$ with $x_1,\ldots,x_{d+e}$ algebraically independent over $\R$,
$z_1$ separable-algebraic over $\R(x_1,\ldots,x_d)$, and $z_2$
separable-algebraic over $\R(x_{d+1},\ldots,x_{d+e})$. Then
$F=\R(x_1,\ldots,x_{d+e},z_1,z_2)$. Let $G_1\in k[x_1,\ldots,x_d,Z_1]$
be the irreducible polynomial of $z_1$ over $k(x_1,\ldots,x_d)$ and
$G_2\in k[x_{d+1},\ldots, x_{d+e}, Z]$ be the irreducible polynomial of
$z_2$ over $k(x_{d+1},\ldots,x_{d+e})$.

Take non-zero elements $f_1,\ldots,f_m\in F$ such that $U:=
H'(f_1)\cap\ldots\cap H'(f_n)\ne\emptyset$. Take $\zeta\in U$ and write
\[
f_i(x_1,\ldots,x_{d+e},z_1,z_2)\>=\>
\frac{g_i(x_1,\ldots,x_{d+e},z_1,z_2)}{h_i(x_1,\ldots,x_{d+e})}
\]
with polynomials $g_i\in\R[X_1,\ldots,X_{d+e},Z_1,Z_2]$ and $h_i\in
\R[X_1,\ldots,X_{d+e}]$. Choose an ordering on $F$ compatible with
$\zeta$. Then the existential sentence

\begin{eqnarray*}
\lefteqn{\exists X_1\ldots\exists X_{d+e} \exists Z_1 \exists Z_2:}&&\\
 & G_1(X_1,\ldots,X_d,Z_1)\>=\>0\;\wedge\;
\frac{\displaystyle\partial G_1}{\displaystyle\partial Z_1}
(X_1,\ldots,X_d,Z_1)\>\ne\>0\>\wedge  & \\
 & G_2(X_{d+1},\ldots,X_{d+e},Z_2)\>=\>0\;\wedge\;
\frac{\displaystyle\partial G_2}{\displaystyle\partial Z_2}
(X_{d+1},\ldots,X_{d+e},Z_2)\>\ne\>0\>\wedge & \\
 & \bigwedge_{1\leq i\leq m} h_i(X_1,\ldots,X_{d+e})\ne 0
\>\wedge\> \frac{g_i(X_1,\ldots,X_{d+e},Z_1,Z_2)}
{h_i(X_1,\ldots,X_{d+e})} >0 &
\end{eqnarray*}
holds in $F$ with this ordering. By Tarski's Transfer Principle, it also
holds in $\R$ with the usual ordering. That is, there exist $a_1,\ldots,
a_{d+r},b_1,b_2\in\R$ such that
\begin{eqnarray}
 & G_1(a_1,\ldots,a_d,b_1)\>=\>0\;\wedge\;  \label{KP1}
\frac{\displaystyle\partial G_1}{\displaystyle\partial Z_1}
(a_1,\ldots,a_d,b_1)\>\ne\>0 & \\
 & G_2(a_{d+1},\ldots,a_{d+e},b_2)\>=\>0\;\wedge\;  \label{KP2}
\frac{\displaystyle\partial G_2}{\displaystyle\partial Z_2}
(a_{d+1},\ldots,a_{d+e},b_2)\>\ne\>0 & \\
 & \bigwedge_{1\leq i\leq m} h_i(a_1,\ldots,a_{d+e})\ne 0
\>\wedge\> \frac{g_i(a_1,\ldots,a_{d+e},b_1,b_2)}
{h_i(a_1,\ldots,a_{d+e})} >0 &
\end{eqnarray}
Hence for every $\R$-place $\zeta\in M(F)$ such that $\zeta (x_i)=a_i$
and $\zeta (z_j)=b_j$ we will have that $\zeta (f_i)>0$, $1\leq i\leq
m$. By Lemma~\ref{KP}, (\ref{KP1}) guarantees that there is $\zeta_1\in
M(F_1)$ such that $\zeta_1(x_i)=a_i\,$, $1\leq i\leq d$, and $\zeta_1
(z_1)= b_1\,$, and (\ref{KP2}) guarantees that there is $\zeta_2\in
M(F_2)$ such that $\zeta_2(x_i)=a_i\,$, $d+1\leq i\leq d+e$, and
$\zeta_2(z_2)= b_2\,$. Consequently, there is $\zeta\in \im(\iota)$ with
$\zeta (x_i)=a_i$ and $\zeta (z_j)=b_j\,$. It follows that $\zeta\in
U\cap\im (\iota)$. This proves that the image of our construction lies
dense in $M(F)$.

From Remark~\ref{many} it again follows that there are infinitely many
$\R$-places $\zeta$ of $\R(x_1,\ldots,x_{d+e})$ such that $\zeta(x_i)=
a_i\,$. These places can be extended to $F$ by setting $\zeta(z_j)=
b_j\,$. All of them have archimedean value group. In contrast, all
places in $\im(\iota)$ are compositions of two non-trivial places and
therefore have non-archimedean value group. This shows that
$U\setminus\im(\iota)$ is infinite.
\end{proof}

As before, one proves:
\begin{corollary}
If $F_1|\R$ and $F_2|\R$ are function fields, then a compatible
embedding cannot be continuous with respect to the product
topology on $M(F_1)\times M(F_2)$.
\end{corollary}
%

%

%
%
\section{Raising the transcendence degree}  \label{sectrp}
In this final section, we show how to use previous constructions to
embed $M(K)$ in $M(L)$, for an arbitrary field $K$ and suitable
transcendental extensions $L$ of $K$.

\begin{theorem}                             \label{rpemb}
Assume that $L$ admits a $K$-rational place $\xi$. Then
\[
\iota:\>M(K)\ni\zeta\>\mapsto\>\zeta\circ\xi\in M(L)
\]
is a continuous embedding compatible with restriction.
\end{theorem}
\begin{proof}
It is clear that the embedding is compatible with restriction. For the
continuity, take $f\in L$ and assume that $H'(f)\cap \im(\iota)\ne
\emptyset$. Pick $\zeta\in M(K)$ such that $\zeta\circ\xi=\iota(\zeta)
\in H'(f)$. It follows that $(\zeta\circ\xi)(f)\ne\infty$ and therefore,
$\infty\ne\xi(f)\in K$. For arbitrary $\zeta\in M(K)$, we have that
$(\zeta\circ\xi)(f)=\zeta(\xi(f))$, so $\zeta\circ\xi\in H'(f)
\Leftrightarrow \zeta\in H'(\xi(f))$. Hence, $\iota^{-1}(H'(f))=
H'(\xi(f))$, which proves that $\iota$ is continuous.
\end{proof}

There are fields $L$ of arbitrary transcendence degree over $K$ which
allow a unique $K$-rational place $\xi$. This fact has been used in
\cite{eo} to show that a given space of $\R$-places can be realized over
arbitrarily large fields. The other extreme is:

\begin{corollary}
Take a collection $x_i\,$, $i\in I$, of elements algebraically
independent over $K$. Then there are at least $|K|^{|I|}$ many distinct
continuous embeddings of $M(K)$ in $M(K(x_i\mid i\in I))$, all of them
compatible with restriction and having mutually disjoint images.
\end{corollary}
This follows from the fact that for every choice of elements $a_i\in K$
there is a $K$-rational place $\xi$ of $L$ such that $\xi(x_i)=a_i\,$.

\begin{corollary}
There are at least $2^{\aleph_0}$ many continuous embeddings of
$M(\R(x))$ in $M(\R(x,y))$, all of them compatible with restriction
and having mutually disjoint images.
\end{corollary}

It should be noted that Theorem~\ref{ce} does not follow from
Theorem~\ref{rpemb}. The condition that $vR$ is a convex subgroup of
$vF$ does by no means imply that $F(y)$ admits a $R(y)$-rational place.

\bn

\end{document}